\documentclass[11pt,reqno]{amsart}
\usepackage{amssymb, amsmath}
\usepackage{mathrsfs}
\usepackage{enumerate}

\newcommand {\PP}{{I\kern-.3em P}}
\newcommand {\ZZ}{{\mathbb{Z}}}

\newcommand {\HH}{{\mathbb{H}}}
\newcommand {\RR}{{\mathbb{R}}}

\newcommand{\beq}{\begin{equation}}
\newcommand{\eeq}{\end{equation}}
\newcommand{\beqq}{\begin{equation*}}
\newcommand{\eeqq}{\end{equation*}}

\newcommand{\charf}{\raisebox{\depth}{\(\chi\)}}
 
\def\I1{\mathbb{1}}
\def\d{\displaystyle}

\newtheorem{thm}{Theorem}
\newtheorem{lemma}{Lemma}
\newtheorem{prop}{Proposition}
\newtheorem{cor}{Corollary}
\newtheorem{rem}{Remark}
 \numberwithin{equation}{section}

\title[Geodesic Returns]{Excursion and return times of a geodesics   to a subset  of a hyperbolic  
Riemann  surface}
\author{Andrew Haas }
\begin{document}

\maketitle
%  \markboth {Andrew Haas}{Geodesic excursions}
\begin{abstract}

 We calculate the asymptotic average rate at which a generic geodesic on a finite area hyperbolic 2-orbifold returns to a subsurface with geodesic boundary. As a consequence we get the  average time a generic geodesic spends in such a subsurface. Related results are obtained for excursions into a collar neighborhood of a simple closed geodesic and the associated distribution of excursion depths. 

 \end{abstract}

\section{Introduction}
The  geodesic flow on the unit tangent bundle of a finite area hyperbolic Riemann surface is ergodic \cite{nicholls}.   One important, well known consequence  is that on average, the time spent by a generic geodesic  in a subset of the surface is equal to the relative area of the set. At first glance it does not appear that this approach tells us anything directly about other specific aspects of the behavior of the geodesic relative to the  set; for example, about the average rate at which the geodesic returns to the set, the average length of each visit to the  set ---called the excursion time-- or the average time between visits: all related quantities.  The reason for this is that these values are determined by certain aspects of the geometry of the   sets  beyond just  the relative area. 
 Nevertheless, for a reasonably large and interesting selection of sets, once the geometry is properly accounted for,  it is the classical Ergodic Theorem for flows from which one can infer the existence of a limit and compute its value.

In the cases considered here,  where    the subset is a subsurface  with geodesic boundary  or a collar  neighborhood of a simple closed geodesic, it is possible to provide precise descriptions of these types of  behavior for a generic geodesics relative to the set.  In the first instance the return time depends only on the length of the boundary and the area of the surface, whereas excursion times are determined by the relative area along with the length of the boundary. In the second case the return time depends on the geodesics length, the width of its collar  and the area of the surface, while excursion times only depend on the width of the collar.  This last fact is at first surprising, yet it mirrors the result from \cite{haas} where it was shown that the average excursion time into an embedded cusp neighborhood is $\pi$, independent of the area of the neighborhood and furthermore, independent of the surface. In the latter case it is also possible to describe the distributions of the maximal depths of geodesic excursions--again something that depends only on the geometry of the collar. This is reminiscent of results from \cite{bjw} and \cite{haas3}.

The return time to a cusp  relative to the depth of the excursion was studied by Sullivan  in \cite{sullivan}, leading in various directions, to generalizations  and  tangents. 
 One of these tangents was followed by Nakada in \cite{nakada}, where he found the average return time to a cusp neighborhood of a finite volume 3-manifold, as one step in his study of approximation properties of rationals in an imaginary quadratic number field. His approach was taken up by Stratmann \cite{strat}, who used it to get estimates for the average return time of a generic geodesic in a hyperbolic manifold--where generic  means with respect to the  Liouville-Patterson measure on the unit tangent bundle. We employed their methods in \cite{haas} to   determine average return and excursion times to a cusp neighborhood on a finite area hyperbolic   surface and, like Nakada and Stratmann, we used these values to establish metrical results for approximation by the cusps of a Fuchsian group.   

\section{Background and main results}
\subsection{Geodesic boundary}
A finite area hyperbolic 2-orbifold $S$ is the quotient of the Poincar\'e upper-half plane $\HH$ by a discrete group $G\subset \text{PSL}_2(\RR).$ 

Let $\gamma_1,\ldots, \gamma_n$ be a collection of  disjoint simple closed geodesics on $S$, which together bound  a subsurface $M$ of $S$. We write $\Gamma =\cup\gamma_i$ for the boundary of $M$. The length of a closed geodesic $\beta$ is written $l(\beta).$ The total length of the boundary of $M$ is then $l(\Gamma)=\sum l(\gamma_i).$

Each vector $v$ in the unit tangent bundle $T_1S$ uniquely determines a geodesic ray $\alpha_v:[0,\infty)\rightarrow S$ which is the projection to $S$  of the forward orbit $\{G_t(v)\,|\,t\in [0,\infty)\},$ of the geodesic flow on the unit tangent bundle. If $\alpha_v:=\alpha_v([0,\infty))$ intersects the boundary $\Gamma$ of $M$ infinitely often, then there is a sequence of pairs of parameter values $\{(t_i,s_i)\}_{i=1}^{\infty}$ with $t_i<s_i<t_{i+1}$, so that 
\beq\label{alpha}
\bigcup_{i=1}^{\infty}\alpha_v[(t_i,s_i)]=\alpha_v\cap M.
\eeq
In other words, $\alpha_v$ meets the set $M$ in precisely the arcs $\alpha_v[(t_i,s_i)]$ of $\alpha_v.$ We shall refer to these arcs as the excursions of $\alpha_v$ into $M$ and to the sequence of parameters as the excursion parameters of $\alpha_v.$
  Let $\#X$ denote the cardinality of the set $X$. $\text{area}(T)$ shall denote the hyperbolic area of $T\subset S$.
 Area is measured in the hyperbolic metric.  

Then the asymptotic average rate at which a geodesic returns to  the set $M$ is given as follows.

\begin{thm}\label{return0}
For almost all $v\in T_1S$  the excursion parameters of $\alpha_v$ satisfy
\beq\label{average}
\d{ \lim_{t\rightarrow\infty}\frac{1}{t} \#\{ n\,|\, t_n<t\}=\frac{l(\Gamma)}{\pi\, \text{\em area} (S)}.}\hskip 1.7in
\eeq
\end{thm}

 Since geodesics are parameterized by arc length, the length of the arc $\alpha[(t_i,s_i)]$ is $s_i-t_i$. One consequence of the theorem is  the value for the asymptotic average  length of an excursion.  

\begin{cor}\label{return1}
For almost all $v\in T_1S$,  
 \beq\label{part1}
  \lim_{n\rightarrow\infty}\frac{1}{n}\sum_{i=1}^{n} (s_i-t_i)= \frac{\pi\,\text{\em area} (M)}{l(\Gamma)}.\hskip 1.5in
\eeq
\end{cor}

\begin{rem}
{\em The average distance between the starting points of the first $n+1$ excursions is computed by the sum 
\beqq
\frac{1}{n}\sum_{i=1}^n t_{i+1}-t_{i}=\frac{t_{n+1}-t_1}{n}.
\eeqq
Since $n$ is essentially $\#\{ i\,|\, t_i\leq t_n\}$, by   taking the limit as $n\rightarrow\infty$, we see that the average distance between consecutive excursions is generically equal to the inverse of the value for the limit in (\ref{average}), }$\d{\frac{\pi\, \text{area} (S)}{l(\Gamma)}.}$
\end{rem}

\subsection{Collar neighborhoods}
Let $\gamma$ be a simple closed geodesic on $S.$    Let $\bf{C}_r(\gamma)=\bf{C}_r$ be the   collar neighborhood of  $\gamma$. This is the set of points within a fixed distance $r$ of $\gamma$. If $r$ is not too large ($r<R_{l(\gamma)}=\log\coth\frac{l(\gamma)}{4}$, see Section \ref{norm}), the $r$-collar about $\gamma$ is divided by $\gamma$ into disjoint open cylinders denoted   $A_r$ and $B_r$, which we will call half-collars. If a subsurface $M$ is specified, then we shall suppose that $A_r \subset M.$ Let $\lambda_{A_r}=\partial{\bf C_r}\cap\partial A_r$ and  $\lambda_{B_r}=\partial{\bf C_r}\cap\partial B_r $, where $\partial D$ denotes the boundary of $D$ in $S$.
The values $r$, $l(\gamma)$ and area($\bf{C}_r$) are not independent but rather,  the width of the collar and the length of $\gamma$ determine one another and the two of them determine the area of the collar.  This will be made precise in Section \ref{phi}, Proposition \ref{prop2}.

We shall suppose the geodesic ray  $\alpha_v$, determined by the direction $v$ in $T_1S$, intersects $\bf{C}_r$ infinitely often. Let $P_r=\{(t_i,s_i)| i\in\ZZ^+\}$ and $Q_r=\{(\rho_i,\eta_i)| i\in\ZZ^+ \}$ be sequences of pairs of parameter values  with $t_i<s_i<t_{i+1}$, $\rho_i<\eta_i<
\rho_{i+1}$, so that $\alpha_v(t_i)\in\lambda_{A_r}$, $\alpha_v(\rho_i)\in\lambda_{B_r}$, $\{ \alpha_v(s_i),\alpha_v(\eta_i)\}\subset\{\lambda_{A_r},\lambda_{B_r}\}$ and 
\beq\label{alpha}
\bigcup_{P_r}\alpha_v[(t_i,s_i)]\cup\bigcup_{Q_r}\alpha_v[(\rho_i,\eta_i)]=\alpha_v\cap\bf{C}_r.
\eeq
In other words, these parameters determine the segments of $\alpha_v$ in $\bf{C}_r$  with both endpoints on its boundary. We will focus on those that enter $\bf{C}_r$ at $\lambda_{A_r}$. $P_r$ is called the set of excursion parameters. 
Let $P'_r=\{(t'_i,s'_i)\}$ be the subset of excursion parameters, reindexed by consecutive numbers,  where both values  $\alpha_v(t'_i)$ and $\alpha_v(s'_i)$  lie in $\lambda_{A_r}$.

\begin{thm}\label{collar0}
For $r<R_{l(\gamma)}$ and almost all $v\in T_1S$  
\beq\label{collar1}
 {\lim_{t\rightarrow\infty}\frac{1}{t}\#\{ n\,|\, t_n<t\} =\frac{l(\gamma)\cosh r }{\pi\,\text{\em area}(S)} =\frac{\sqrt{l(\gamma)^2+(\frac{1}{2}\text{\em area}({\bf C_r}))^2}}{ \pi\, \text{\em area} (S)}}  \hskip  .7in
  \eeq
 \beq\label{collar2}
 {\lim_{t\rightarrow\infty}\frac{1}{t}\#\{ n\,|\, t'_n<t\} =\frac{l(\gamma)(\cosh r -1)}{\pi\,\text{\em area}(S)} =
 \frac{\sqrt{l(\gamma)^2+(\frac{1}{2}\text{\em area}({\bf C_r}))^2}-l(\gamma)}{ \pi\, \text{\em area} (S)}}    
 \eeq
 \end{thm}
 
 In this case the analogue of Corollary \ref{return1}, the asymptotic average length of an excursion into a collar neighborhood, is given by the following. 
 
 \begin{cor}\label{cor2part1}
 For $r<R$ and almost all $v\in T_1S$ 
  \beq
  \lim_{n\rightarrow\infty}\frac{1}{n}\sum_{i=1}^{n} (s_i-t_i)= 2\pi\tanh r=\frac{2\pi\, \text{\em area}(\bf{C}_r)}{\sqrt{l(\gamma)^2 + \text{\em area}(\bf{C}_r)^2}}.  \hskip .5in
\eeq
 
\end{cor}

 \subsection{The distribution of collar excursions}

In \cite{haas3} the number theoretic distributions from the Doeblin-Lenstra Conjecture {\cite{bjw}},{\cite {dk}} and the related Theorem of Bosma 
\cite{bosma} were reinterpreted and expanded in a geometric setting. The analogous notions can be defined in terms of the distribution of the depths of maximal penetration of geodesic excursions into a collar neighborhood of a simple closed geodesic. Given a geodesic $\alpha_v$ with excursion parameters $P_r$,  define $D_v(i)=\inf\{\text{ dist}(\alpha_v(t),\gamma)\,|\, t_i<t<s_i\}$, where $\text{ dist}(c,\gamma)$ is the distance between the point c and the geodesic $\gamma$. 
This is the depth of the $i^{th}$ excursion. Using Theorem \ref{collar0} we  show, in the next corollary, that the depths are nicely distributed, in a fashion that is independent of $l(\gamma)$. Choose  $R_0< R_{l(\gamma)}$. 

\begin{cor}\label{bosma}
For $r\leq R_0$ and almost all $v\in T_1S$ the limit 
\beq
\lim_{n\rightarrow\infty}\frac{1}{n}\#\{i\,|\, 1\leq i \leq n, D_v(i)\leq r\}
\eeq 
is defined and takes the value
\beq
\delta_v(r)=\frac{\cosh r}{\cosh R_{0}}.
\eeq
\end{cor}

 \subsection{The invariant measure}

The unit tangent bundle of the hyperbolic plane $\HH$ can be given as a cartesian product $T_1\HH=\HH\times S^1.$ In these coordinates the natural invariant measure for  the geodesic flow $\tilde{G}$ has the form $\tilde{m}=dAd\phi.$ There is another set of very useful coordinates in which $T_1\HH $ is described, up to measure zero,  as the the set of triples $(x,y,t)\in\RR^3$ with $x\not=y.$ Then $(x,y,t)$ corresponds to the vector $v=\dot{\alpha}(t)\in T_1\HH,$ where  $\alpha$ is the geodesic in $\HH$ oriented from the  endpoint  $\alpha_{-}=x $ to $\alpha_{+}=y$  and parameterized so that $\alpha(0)$ is the Euclidean midpoint of the semicircle $\alpha(\RR)$. The geodesic flow on $T_1\HH$ satisfies $\tilde{G}_s(x,y,t)=(x,y,t+s).$ Furthermore, the geodesic flow on $T_1S$ has the invariant probability measure $\mu$, whose lift to $T_1\HH$ is equal to  \cite{nicholls}
\beq
\tilde{\mu}= \frac{1}{2\pi\,\text{area}(S)}\tilde{m}=\frac{1}{\pi\, \text{area}(S)}(x-y)^{-2}\, dx dy dt.
\eeq

Let $\mathcal{E}\subset T_1S$ be the set of vectors $v$ for which the orbit $G_t(v),\, t>0$ is dense in $T_1S$. 
 
 As a consequence of  the Ergodic Theorem and the Poincar\'e Recurrence Theorem,  \cite{sinai}, $\mathcal{E}$ has full measure in $T_1S$.

 \subsection{The normalization of G}\label{norm}
 
 Given a simple closed geodesic $\gamma$ on $S$ we may suppose that the Fuchsian group $G$ has been normalized so that the imaginary axis $\mathcal{I} $ in $\HH$ covers $\gamma$. Then the stabilizer of $\mathcal{I} $ in $G$ is generated by the transformation $g(z)=\zeta z$ with   $\log \zeta = l(\gamma)$. The arc $\tilde{\gamma}=\{it\,|\, 1\leq t < \zeta\}\subset\mathcal{I}$ is mapped injectively onto $\gamma$ by the covering projection. 
 
 $\bf{C}_r(\gamma)$ is the image of the $r$-neighborhood  $\bf{C}_r(\mathcal{I})$ of $\mathcal{I} $  under the covering projection.    It follows from   the  Collar Lemma (see \cite{ara} and \cite{buser} for the many references)
  
 that there is a value $R_{l(\gamma)}=\log\coth \frac{l(\gamma)}{4}$ so that for  $r< R_{l(\gamma)},$ the collar $ \bf{C}_r(\mathcal{I})$ is mapped disjointly from itself by any $h\in G$ which is not a power of $g$.  Consequently,  $\bf{C}_r(\gamma)$ is a cylinder embedded in $S$, as we have been assuming. Henceforth, we shall suppose that $r<R_{l(\gamma)}.$
  
 The curves $\lambda_{A_r}$ and $\lambda_{B_r}$ lift to straight lines $\tilde{\lambda}_{A_r}$ and $\tilde{\lambda}_{B_r}$ bounding $\bf{C}_r(\mathcal{I})$. They eminate from the origin and make an angle $\phi$ with   $\mathcal{I} $.       We shall further stipulate that the lift of the half-collar $A_r$ in  $\bf{C}_r(\mathcal{I})$ lies in the right half-plane.

\section{Proofs for regions with geodesic boundary}\label{boundary}
\subsection{The single geodesic}\label{boundarya}

Let $\mathcal{L}^*(\gamma)=\mathcal{L}^*(\gamma, A_r)$ be the subset of the unit tangent bundle over   $\gamma$ where  $v\in\mathcal{L}^*(\gamma)\,$ if there exists $\tau_0$ so that  for  $0<t<\tau_0$ the points of $\alpha_v(t)$ lie in $A_r$.  This is the subset of vectors based on $\gamma$ that point into $A_r$. Note that the set does not depend on the choice of $r<R_{l(\gamma)}$. Then $\mathcal{L}(\gamma)=\mathcal{L}^*(\gamma)\cap\mathcal{E}$ is a cross-section for the geodesic flow on $T_1S$, \cite{bks}. In other words, for almost all $v\in T_1S$ there exists an increasing sequence of values $\tau_i$ so that $G_{\tau_i}(v)\in \mathcal{L}(\gamma)$.

Given $\varepsilon<R_{l(\gamma)}-r$, define the $\varepsilon$-thickened section $ \mathcal{L}_{\varepsilon}(\gamma)=\{G_t(v)\,|\, t\in [0,\varepsilon], v\in \mathcal{L}(\gamma)\}.$ Analysis of the thickened section  is the main tool in the proof of Theorem \ref{return0}.

\begin{prop}\label{thick}
\beq
\mu( \mathcal{L}_{\varepsilon}(\gamma))=\frac{\varepsilon\, l(\gamma)}{\pi\, \text{\em area} (S)}
\eeq
\end{prop}

\begin{proof}
 
The cross-section $\mathcal{L}(\gamma)$ and the thickened section  $ \mathcal{L}_{\varepsilon}(\gamma)$ lift to  the subsets $\mathcal{L}(\tilde{\gamma})$ of $T_1\HH$ over $\tilde{\gamma}$ and its thickened section $ \mathcal{L}_{\varepsilon}(\tilde{\gamma})$, respectively.  Furthermore, the projection from $\mathcal{L}_{\varepsilon}(\tilde{\gamma})$ to  $ \mathcal{L}_{\varepsilon}(\gamma)$ is injective.  In this way we reduce the computation in Proposition \ref{thick} to a calculation in $(x,y,t)$-coordinates in $\RR^3$ of $\tilde{\mu}(\mathcal{L}_{\varepsilon}(\tilde{\gamma}))=\mu( \mathcal{L}_{\varepsilon}(\gamma)).$
\vskip .1in

For a given $x\in \RR^+$ and $t>0$, the point $y$ at which the geodesic $\overline{xy}$ with endpoints $x$ and $y$  meets the point $it\in \mathcal{I},$ is the solution to the equation

\beq\label{repeat}
|it-\frac{x+y}{2}|^2=(\frac{x-y}{2})^2.
\eeq
This is ${y=-\frac{t^2}{x}}.$ Therefore for a given $x>0$, $I_x=[-\frac{\zeta^2}{x},-\frac{1}{x})$ is  the interval of values $y\in  \RR $ for which  $\overline{xy}$ intersects the interval $\tilde{\gamma}$.

For $x\in \RR$ and $y\in I_x$, let $t_{xy}$ denote the parameter value for which the unit tangent vector   $(x,y,t_{xy})\in  \mathcal{L}(\tilde{\gamma}).$
 
Then
\beq\label{need}
 \mathcal{L}_{\varepsilon}(\tilde{\gamma}) =\{ (x,y,t)\,|\,x\in\RR,\, y\in I_x\, \text{and}\, t_{xy}-\varepsilon \leq t\leq t_{xy }  \}.
\eeq

Consequently, 
\beq\label{area}
 \tilde{\mu}(\mathcal{L}_{\varepsilon}(\tilde{\gamma}))=\frac{1}{\pi\,\text{area}(S)}\int_{\RR^+}\int_{I_x}\int_{t_{xy}-\varepsilon }^{t_{xy} }(x-y)^{-2}\,dtdydx  
 \eeq
 \beqq
 \,\,\,\,\,\,\,=\frac{\varepsilon}{\pi\,\text{area}(S)}\int_{0}^{\infty}\int_{-\frac{\zeta^2}{x}}^{ -\frac{1}{x} }(x-y)^{-2}\,dydx.  
\eeqq
\beqq
 \,\,\,=\frac{\varepsilon}{ \pi\,\text{area}(S)}\log \zeta =\frac{\varepsilon }{ \pi\,\text{area}(S)}l(\gamma)
   \eeqq
\end{proof}\subsection{Many geodesics}

We shall prove a theorem that is slightly more general than Theorem \ref{return0}.
Let $\{\gamma_i\}_{i=1}^n$ be a finite sequence of  mutually disjoint simple closed geodesics, except that we allow   geodesics to appear twice in the list. To each $\gamma_i$ in the sequence, associate a distinct half-collars $ A_i.$ With $\Gamma$ as before and $\mathcal{A} =\cup A_i$, define 
$\mathcal{L}(\Gamma, \mathcal{A})=\cup\mathcal{L}(\gamma_i,  A_i) $ and $\mathcal{L}_{\varepsilon}(\Gamma, \mathcal{A})=\cup\mathcal{L}_{\varepsilon}(\gamma_i,  A_i)$. In this case there is not necessarily a subsurface $M$ with which the half-collars are associated.
 
Suppose $v\in \mathcal{E}$. Then $G_t(v)$ will meet $\mathcal{L}(\Gamma, \mathcal{A})$ in a sequence of points $G_{\tau_j}(v).$ In other words, $\alpha_v(\tau_j)$ will lie on one of the $\gamma_i$ with a tangent pointing into $A_i.$ 
Let $\mathcal{N}_v(t) $ denote the number of times the orbit lies in $\mathcal{L}(\Gamma, \mathcal{A})$ or the number of times $\alpha_v$ crosses one of the geodesics in the direction stipulated by a half-collar.

\begin{thm}\label{game}
For almost all $v\in T_1S$  
\beq\label{average2}
\d{\lim_{t\rightarrow\infty}\frac{ \mathcal{N}_v(t)}{t}=\frac{l(\Gamma)}{\pi\, \text{\em area} (S)}.}
\eeq

\end{thm}
\begin{proof}
The characteristic function of a set $Y$ is written $\charf_Y$. For $\varepsilon$ small we have the inequalities
\beq\label{for proof second thm}
\int_0^t\charf_{\mathcal{L}_{\varepsilon}(\Gamma, \mathcal{A})}(G_{\tau}(v))d\tau\,-2\varepsilon\,\leq\, \varepsilon\, \mathcal{N}_v(r)(t)\,\leq\,
  \int_0^t\charf_{\mathcal{L}_{\varepsilon}(\Gamma, \mathcal{A})}(G_{\tau}(v))d\tau\,+2\varepsilon.
  \eeq
Divide through by $t\varepsilon$ and let $t\rightarrow\infty$. By the Ergodic Theorem and Proposition \ref{thick}, for almost all $v\in T_1S$, the left and right hand limits  converge to 
\beqq
 (1/\varepsilon)\lim_{t\rightarrow\infty}\frac{1}{t}\int_0^t\charf_{\mathcal{L}_{\varepsilon}(\Gamma, \mathcal{A})}(G_{\tau}(v))d\tau=
  (1/\varepsilon)\sum_{i=1}^n \lim_{t\rightarrow\infty}\frac{1}{t}\int_0^t\charf_{\mathcal{L}_{\varepsilon}(\gamma_i,  A_i)}(G_{\tau}(v))d\tau 
 \eeqq
 \beqq
 =(1/\varepsilon)\sum_{i=1}^n\mu(\mathcal{L}_{\varepsilon}(\gamma_i )=\frac{\sum_{i=1}^n  l(\gamma_i)}{\pi\, \text{ area} (S)}= \frac{  l(\Gamma)}{\pi\, \text{  area} (S)}.
 \eeqq
\end{proof}

\noindent {\em Proof of Theorem \ref{return0}}.
The theorem follows from Theorem \ref{game} by choosing $\Gamma$ to be the boundary of the $M$ and taking the half-collar $A_i$ associated to $\gamma_i$ to be the  one inside $M$.
\qed
\vskip .1in

\subsection{Excursion times} 
 
\noindent {\em Proof of Corollary \ref{return1}}.\, 
With the possible exception of the very first, the sum of the first $n$ excursion lengths is
\beqq\label{flow}
\sum_{i=1}^{n} (s_i-t_i)=\int_0^{s_n}\charf_{T_1M}(G_{\tau}(v))\,d\tau
\eeqq
The average can then be written in the form 
\beqq
\frac{1}{s_n}\int_0^{s_n}\charf_{T_1M}(G_{\tau}(v))\,d\tau\,\times\, 
 \frac{s_n}{n}\\
\eeqq
The value $n$ on the right  is $\#\{i\,|\, t_i<s_n\},$ from Theorem \ref{return0}. By that theorem, the limit as $n\rightarrow\infty$ of $s_n/n$ exists and is the inverse of the right-hand side of equation (\ref{average}). By the Ergodic Theorem the limit of the first factor also exists and is the  $\mu$-measure of $T_1M$; this is, the area of $ M$ divided by the area of $ S$. We should note that in both instances convergence is almost everywhere. Taken together  
 
this completes the proof.
\qed

\section{Proofs for collars}

\subsection{Preliminaries}\label{phi}
Recall that $\phi$ is the angle made by the ray $\tilde{\lambda}_A$ and the imaginary axis.  Write $\sin\phi +i\cos\phi=a+ib=p_{\phi}$. The relationship between area$(\bf{C}_r),$ $ r$ and $\phi$ is given in the following proposition.
 
 \begin{prop}\label{prop2}
 $\text{\em area}({\bf C_r}(\gamma))=2l(\gamma)\tan\phi=2l(\gamma)\sinh r.$ 
  In particular,\\
   $\d{\cosh r= \frac{1}{\cos\phi}=\frac{1}{b} }$.
 \end{prop}
 
 \begin{proof}
 In order to see how $r$ and $\phi$ are related, we look at the geodesic segment that lies on the unit circle centered at 0 running  between the imaginary axis and the line $\tilde{\lambda}_A$.  Using polar coordinates with the hyperbolic metric in $\HH$ we see that
 \beqq
r=\int_{\frac{\pi}{2}-\phi}^{\frac{\pi}{2}} \frac{d\theta}{\sin\theta}=\log (\sec\phi+\tan\phi).
 \eeqq
  Therefore, $\d{e^{r}=\sec\phi + \tan\phi = \sqrt{1+\tan^2\phi} + \tan\phi.}$ Solving for $\tan\phi$ gives $\sinh r.$

The area of  ${\bf C}_{r}(\gamma)$ is the area of the region in $\bf{C}_r(\tilde{\gamma})$ between the two circles centered at the origin with radii 1 and $\zeta$ respectively. Again, computing with the hyperbolic metric in polar coordinates we have
\beqq
\text{area}({\bf C}_{r}(\gamma))=2\int_1^{\zeta}\frac{1}{r}\, dr\, \int_{\frac{\pi}{2}-\phi}^{\frac{\pi}{2}}\frac{1}{\sin^2\theta}\, d\theta = 2\log\zeta\tan\phi=2l(\gamma)\tan\phi.
\eeqq

 \end{proof}

%\subsection{}
The following simple, geometric lemma will be very useful.

\begin{lemma}\label{limits}
\begin{enumerate}
\item
Given $x>0$, if the geodesic $\overline{xy}$ in $\HH$   passes through $x$ and the point $tp_{\phi}$, $t>0$, then 
\beq\label{one}
y=\frac{a x t-t^2}{x-at}. 
\eeq
\item
The geodesic $\overline{xy}$ in $\HH$, $x>y$, tangent to $\tilde{\lambda}_A$ at the point $tp_{\phi}$ has endpoints 
\beq \label{second}
x=t(\frac{1+b}{a})\,\,\, \text{and}\,\,\,
 y=t(\frac{1-b}{a}).
\eeq
Consequently, given $x$, $\overline{xy}$ it tangent to $\tilde{\lambda}_A$ when 
\beq\label{third}
y=(\frac{1-b}{1+b})x.
\eeq
\end{enumerate}
\end{lemma}

\begin{proof}
As in formula (\ref{repeat}) the geodesic  $\overline{xy}$ contains $tp_{\phi}$ if
\beqq
|\frac{x+y}{2}-tp_{\phi}|^2=(\frac{x-y}{2})^2.
\eeqq
With $p_{\phi}=a+ib$, we square both sides and simplify to get $xy-tax-tay+t^2=0.$ Solving for $y$ gives (\ref{one}).

Suppose that $\overline{xy}$ is tangent to $\tilde{\lambda}_A$ at  $tp_{\phi}$. Let $w$ be the midpoint  on the real axis  between $y$ and $x$. Then  the line from $w$ to $tp_{\phi}$ is orthogonal to $\tilde{\lambda}_A$. Equating slopes we have 
$\frac{tb}{ta-w}=-\frac{a}{b}$ or $w=\frac{t}{a}.$ Then the radius of the semi-circle $\overline{xy}$ is $|tp_{\phi}-\frac{t}{a}|=\frac{b}{a}t.$
With center and  radius in hand it is easy to write down $x$ and $y$ which gives (\ref{second}). Formula (\ref{third}) follows.
\end{proof}

\subsection{The thickened section}

In order to prove Theorem \ref{collar0} we shall recycle the approach taken in the proof of Theorem \ref{return0}, only this time the thickened section will be defined with respect to the boundary $\lambda_{A_r}$ of the collar neighborhood of $\gamma$.  The situation is somewhat more involved.

Let $\bf{s}$ denote the segment of $\tilde{\lambda}_A$ between the points $p_{\phi}$ and $\zeta p_{\phi}$. It is a preimage of $\lambda_{A_r}$ under the covering projection. Henceforth $x$ will alway be a positive number. 

We define a notion of intersection between the directed geodesic $\beta=\overline{xy}$ and the arc $\bf{s}$ that only counts the first point at which the geodesic intersects  $\tilde{\lambda}_A$. More precisely, if $\beta\cap\tilde{\lambda}_A=\emptyset$ then  set $ \overline{xy}\,\hat{\cap}\,\bf{s} =\emptyset.$ And if  $\beta$   intersects $\tilde{\lambda}_A$ in the points $\beta(t_1)$ and $\beta(t_2)$, $t_1\leq t_2$, then  
\beqq
\d{ \overline{xy}\,\hat{\cap}\,\bf{s} = \left\{ \begin{array}{cc} 
 \{\beta(t_1)\} & {\rm if} \,\beta(t_1)\in\bf{s} \\
\emptyset &  {\rm otherwise.}
%{\rm if}\,  x \geq 2.
\end{array} \right.} 
\eeqq

We shall define several sections of the unit tangent bundle over $\lambda_{A_r}$ by specifying  sets of triples $(x,y,t_{xy})$ in the unit tangent bundle over $\bf{s}$. 
  For $i=0,1,2, 3$ let
\beqq
\mathcal{J}^{i*}({\bf s})=\{ (x,y,t_{xy}) | R^i\, \text{and}\, \overline{xy}\,\hat{\cap}\,\bf{s}\not = \emptyset \}
\eeqq
where $R^0,\ldots ,R^3$ denote respectively the conditions, $ R^0:  x>0,y \in\RR,\, R^1:  0<y<x,\, R^2:  0<x<y,\, R^3: y<0<x.$
Given $x,y$, if $ \overline{xy}\,\hat{\cap}\,\bf{s}\not =\emptyset$ then  there is a unique $t_{xy}$ so that  the   point $(x,y,t_{xy})$ lies in the intersection. 

As in section \ref{boundarya}, define $\mathcal{J}^{i}({\bf s})=\mathcal{E}\cap \mathcal{J}^{i*}({\bf s})$ and the thickened sections
\beqq
\mathcal{J}^{i}_{\varepsilon}({\bf s})=\{ (x,y,t)  | (x,y,t_{xy}) \in \mathcal{J}^{i}({\bf s})\, \text{and}\,  t_{xy}\leq t\leq t_{xy}+\varepsilon\}. 
\eeqq
$\mathcal{J}^{i}_{\varepsilon}({\bf s})$ projects to the thickened section $\mathcal{J}^{i}_{\varepsilon}({\lambda_{A_r}})$  in the unit tangent bundle of $S$ over the boundary of the collar. Note that in formula \ref{need} we subtracted $\epsilon$ because $A_r$ was to the right of $\tilde{\gamma}$. Now it is to the left of $\tilde{\lambda}_{A_r}$.

In order to see what these sets represent, note that up to measure zero  $\mathcal{J}^{0}({\lambda_{A_r}})$ is equal to the set of all vectors over $\lambda_{A_r}$ pointing into the collar. $\mathcal{J}^{3}({\lambda_{A_r}})$ is the subset determining geodesics that cross $\gamma$. By elementary hyperbolic geometry such geodesics will then exit the collar at $\lambda_B$. Removing $\mathcal{J}^{3}({\lambda_{A_r}})$ from $\mathcal{J}^{0}({\lambda_{A_r}})$ results in two disjoint sets, which are $\mathcal{J}^{1}({\lambda_{A_r}})$ and $\mathcal{J}^{2}({\lambda_{A_r}})$. Geodesics determined by vectors from these subsets enter $\lambda_{A_r}$ from different sides of $\mathcal{J}^{3}({\lambda_{A_r}})$ and exit the collar at $\lambda_{A_r}$. Together these are all the geodesics that enter and exit via $\lambda_{A_r}$.
%See Figure \ref{}

\begin{prop}\label{last prop}
$\d{\mu(\mathcal{J}^1_{\varepsilon}(\lambda_{A_r})\cup\mathcal{J}^3_{\varepsilon}(\lambda_{A_r})  ) =\frac{\varepsilon(1+  \cosh r) l(\gamma)}{2\pi  \, \text{\em area}(S)}}$

\end{prop}
\begin{proof}
 The computation is done in $(x,y,t)$ coordinates for the sets $\mathcal{J}^{i}_{\varepsilon}({\bf s})$. We shall describe the  limits of integration by  specifying potential $x$ values  and then giving the corresponding set of $y$ values using Lemma \ref{limits}.   Three cases are distinguished for the domain of $x$. In all we have $y<x$. For now we make the additional assumption that $\zeta a< (1+b)/a.$

If $x$ is larger than  $\zeta(1+b)/a$ then for all $y\in\RR$, $\overline{xy}\,\hat{\cap}\,\bf{s}=\emptyset.$ Thus, the first case to consider is when $(1+b)/a< x<\zeta(1+b)/a$. It follows from (\ref{second}) that for such an $x$ there is a point $y$ so that $\overline{xy}$ is tangent to $\tilde{\lambda}_A$ and that this point of tangency lies in   $\bf{s}$.  It also follows from (\ref{third}) that given $x$, the interval of corresponding $y$ values will vary between $(\frac{1-b}{1+b})x$, the point at which $\overline{xy}$ is tangent to $\bf{s}$ and    $\frac{a x \zeta-\zeta^2}{x-a\zeta}$     where  $\overline{xy}$ meets $\zeta p_{\theta} $, the upper endpoint of $\bf{s}$. Thus the measure of the set of vectors in $ \mathcal{L}_{\varepsilon}(\bf{s}) $ determined by these $(x,y)$ values is $M_1$, where
 \beqq
\pi\, \text{area}(S)\, M_1=\int_{\frac{1+b}{a}}^{\zeta(\frac{1+b}{a})}\int^{(\frac{1-b}{1+b})x }_{\frac{a x \zeta-\zeta^2}{x-a\zeta}}\int_{t_{xy}}^{t_{xy}+\varepsilon} (x-y)^{-2}\, dt\, dy\, dx
 \eeqq 
 \beqq
 = \varepsilon\bigg [\frac{1+b}{2b}\log\zeta-\log\zeta-\frac{1}{2}\log \left (\frac{(1+b)^2}{a^2}-2(1+b)+1\right ) 
  \eeqq
  \beqq
  + \frac{1}{2}\log \left ( \frac{(1+b)^2}{a^2}-2(1+b)\zeta+\zeta^2\right )          \bigg ].
 \eeqq
 
The second case is when $\zeta a<x<(1+b)/a$. 
 According to the lemma, as $y$ varies between $\frac{a x-1}{x-a }$ and  $\frac{a x \zeta-\zeta^2}{x-a\zeta},\,\, \overline{xy}\,\hat{\cap}\,\bf{s}$ varies between $p_{\phi}$ and $\zeta p_{\phi}$, taking on all values in  $\bf{s}$. Writing out the integral  as in the previous case, we see that the measure of the corresponding subset of $ \mathcal{L}_{\varepsilon}(\bf{s}) $ is equal to $M_2$ where 
 %\eject
 \beqq
\pi\, \text{area}(S)\, M_2=\varepsilon\bigg [ \frac{1}{2}\log \left (\frac{(1+b)^2}{a^2}-2(1+b)+1\right )-\frac{1}{2}\log\left (a^2\zeta^2-2a^2\zeta+1\right )
\eeqq
\beqq
+\log\zeta+\log (1-a^2)- \frac{1}{2}\log \left ( \frac{(1+b)^2}{a^2}-2(1+b)\zeta+\zeta^2\right )    \bigg ]
\eeqq

The final case is where $a<x<\zeta a$. Given $x$ in this interval, $\overline{xy}$ will meet $\bf{s}$  at $p_{\phi}$ when $y=\frac{a x-1}{x-a }$ and as $y$ goes to $-\infty$, $\,\overline{xy}$ limits at the vertical line intersecting $\bf{s}$ in the point $xp_{\phi}$. In this case the measure of the associated subset of $ \mathcal{L}_{\varepsilon}(\bf{s}) $ is $M_3$ with
\beqq
\pi\, \text{area}(S)\, M_3= \varepsilon[ \frac{1}{2}\log\left (a^2\zeta^2-2a^2\zeta+1\right )-\frac{1}{2}\log(1-a^2)].
\eeqq

Using Proposition \ref{prop2} and the facts $a=\sin\phi$, $b=\cos\phi$ and $\log\zeta =l(\gamma)$, we get $$M_1+M_2+M_3=\frac{\varepsilon(1+b)l(\gamma)}{2\pi b\, \text{area}(S)} =\frac{\varepsilon(1+  \cosh r) l(\gamma)}{2\pi  \, \text{area}(S)},$$
which is
the proposition under the assumption that $\zeta a< (1+b)/a.$ If we reverse the inequality a similar computation yields the same result.
\end{proof}

\subsection{Returns to a collar} 
%\\
\noindent{\em Proof of Theorem 2}.   
To begin we show that 
\beq\label{measure middle section}
\mu(\mathcal{J}^3_{\varepsilon}(\lambda_{A_r}))=\mu( \mathcal{L}_{\varepsilon}(\gamma))=\frac{\varepsilon l(\gamma)}{\pi\, area (S)}.
\eeq

  It was observed in Section  \ref{boundarya} that given $v\in \mathcal{E}$,  $G_t(v)$ will meet $\mathcal{L}(\Gamma, \mathcal{A})$ in a sequence of points $G_{\tau_j}(v), j\in \ZZ^+.$ Each time $\alpha_v$ meets the geodesic $\gamma$ from the $A$ side,   it must have either originated in $A$ or else crossed $\lambda_{A_r}$ first, before proceeding to $\gamma$. Conversely, each time $\alpha_v$ crosses $\lambda_{A_r}$ so that its tangent lies in $\mathcal{J}^3_{\varepsilon}(\lambda_{A_r}),$  it must go on to cross $\gamma$. Thus, for each $j\geq 2$ there will be a corresponding value $\eta_j$ so that $G_{\eta_j}(v)\in \mathcal{J}^3_{\varepsilon}(\lambda_{A_r}) $ and if $G_{\eta}(v)\in   \mathcal{J}^3_{\varepsilon}(\lambda_{A_r}),$ then $\eta=\eta_j$ for some $j$. In other words, the section $\mathcal{J}^3_{\varepsilon}(\lambda_{A_r})$ counts crossing of $\gamma$ from the $A$ side exactly as does $  \mathcal{L}_{\varepsilon}(\gamma)$. 

As a consequence of the above, we can replace $\mathcal{L}_{\varepsilon}(\Gamma, \mathcal{A})$ (with $\Gamma=\gamma$) in formula (\ref{for proof second thm}) by $\mathcal{J}^3_{\varepsilon}(\lambda_{A_r})$. It follows that the value in formula (\ref{average}) of Theorem \ref{return0} is $(1/\varepsilon) \mu( \mathcal{J}^3_{\varepsilon}(\lambda_{A_r}) )=\frac{l(\gamma)}{\pi\, area (S)},$ completing the argument.

The hyperbolic isometry $h(z)=\zeta/\overline{z}$ induces an isometry of $T_1\HH$ that interchanges the  sets $\mathcal{J}^{1}_{\varepsilon}({\bf s})$ and $\mathcal{J}^{2}_{\varepsilon}({\bf s})$. Therefore, we have 
\beqq
\mu(\mathcal{J}^1_{\varepsilon}(\lambda_{A_r})=\mu(\mathcal{J}^2_{\varepsilon}(\lambda_{A_r})\,\,  \text{or}\,\,    \mu(\mathcal{J}^1_{\varepsilon}(\lambda_{A_r})  \cap\mathcal{J}^2_{\varepsilon}(\lambda_{A_r}))= 2\mu(\mathcal{J}^1_{\varepsilon}(\lambda_{A_r})).  
\eeqq

Now making use of the outcomes of formula (\ref{measure middle section}) and Proposition \ref{last prop}, we have
\beqq
\mu(\mathcal{J}^1_{\varepsilon}(\lambda_{A_r})) =\mu(\mathcal{J}^1_{\varepsilon}(\lambda_{A_r})\cup \mathcal{J}^3_{\varepsilon}(\lambda_{A_r}))- \mu(\mathcal{J}^3_{\varepsilon}(\lambda_{A_r}))=\frac{\varepsilon(\cosh r -1) l(\gamma)}{2\pi  \, \text{area}(S)} 
\eeqq
and
\beqq
\mu(\mathcal{J}^0_{\varepsilon}(\lambda_{A_r}))  = 2 \mu(\mathcal{J}^1_{\varepsilon}(\lambda_{A_r}))+ \mu(\mathcal{J}^3_{\varepsilon}(\lambda_{A_r}))=\frac{\varepsilon l(\gamma) \cosh r }{ \pi\, \text{area}(S)}.\eeqq

Let $P=\{(t_i,s_i)\}$ and $P'=\{(t'_i,s'_i)\}$ be the excursion parameters of $\alpha_v$.  For $r>R_{l(\gamma)}$ and $v\in \mathcal{E}$ define the counting functions $N^0_v(r)(t)=\#\{i\,|\, t_i<t\},$ and $N^1_v(r)(t)=\#\{i\,|\, t'_i<t\}.$ 
Taking the role of $\mathcal{L} ( \gamma )$ from the proof of Theorem \ref{return0} the sections $\mathcal{J}^0 (\lambda_{A_r})$ and $\hat{\mathcal{J}}(\lambda_{A_r})=\mathcal{J}^1 (\lambda_{A_r})\cup \mathcal{J}^2 (\lambda_{A_r})$ count crossings by $\alpha_v$ of $\lambda_{A_r}$ that go into $A$ and crossing of $\lambda_{A_r}$ by $\alpha_v$ into $A$ that exit $A$ through $\lambda_{A_r}$, respectively. Thus, rewriting formula (\ref{for proof second thm}) in the first instance gives, for $\varepsilon>0$ sufficiently small,
\beqq
\int_0^t\charf_{ \mathcal{J}^0_{\varepsilon}(\lambda_{A_r})}(G_{\tau}(v))d\tau\,-2\varepsilon\,\leq\, \varepsilon N^0_v(r)(t)\,\leq\,
  \int_0^t\charf_{\mathcal{J}^0_{\varepsilon}(\lambda_{A_r})}(G_{\tau}(v))d\tau\,+2\varepsilon.
  \eeqq
Again, dividing by $t\varepsilon$, letting $t$ go to infinity and applying the Ergodic Theorem proves the first part of Theorem \ref{collar0}.
The proof of the second part of the theorem follows if  $\mathcal{J}^0$ and $N^0$ are replaced by $\hat{\mathcal{J}}$ and $N^1.$
\qed\\
 \\

\noindent{\em Proof of 
Corollary  \ref{cor2part1}}.
Proceeding as in the proof of Corollary \ref{return1}, we have
\beqq
 \lim_{n\rightarrow\infty}\frac{1}{n}\sum_{i=1}^{n} (s_i-t_i) 
 =\lim_{n\rightarrow\infty}\frac{1}{n}\int_0^{s_n}\charf_{T_1M}(G_{\tau}(v))\,d\tau
 \eeqq
 \beqq
=\lim_{n\rightarrow\infty}\frac{1}{s_n}\int_0^{s_n}\charf_{T_1M}(G_{\tau}(v))\,d\tau\,\times\, 
 \frac{s_n}{n}.
 \eeqq
 By the Ergodic Theorem and Theorem \ref{collar0}, this equals
\beqq
 \frac{ \text{area}(\bf{C}_r)}{\text{area}(S)}\,\times\,\frac{\pi\,\text{area}\,(S)}{l(\gamma)\cosh r} =2\pi\tanh r.
 \eeqq
Here, the later equality follows from Proposition \ref{prop2}.\qed
\eject
 \noindent{\em Proof of Corollary \ref{bosma}}.  
 We argue as in \cite{haas}. Write $N_v(r)(t)= \#\{ j\,|\, t_j<t\}$ for the function that counts returns to the radius $r$ collar.  Using Theorem \ref{collar0},  the distribution  $\delta(r)$ can be written as  
\beqq 
\delta(r)= \lim_{n\rightarrow\infty}\frac{ N_v(r)(t)}{N_v(R_0)(t)}=\frac{\cosh r}{\cosh R_0}.
\eeqq
This completes the proof of Corollary \ref{cor2part1}.
 \qed


\begin{thebibliography}{00}
 
 \bibitem{ara}
 A. Basmajian, Constructing pairs of pants, {\em Ann. Acad. Sci. Fenn. Ser. A I Math.}  {\bf15}  (1990),  no. 1, 65--74.

\bibitem{Beardon}
A.F. Beardon,
 {\em The Geometry of Discrete Groups},
  Springer-Verlag, Berlin, 1983.

 \bibitem{bks} {\em Ergodic Theory, Symbolic Dynamics and Hyperbolic
 Spaces}, T. Bedford, H. Keane, C. Series, eds., Oxford Univ. Press, Oxford, 1991.
\bibitem{bosma}  W. Bosma, Approximation by mediants, {\em Math. of Computation,}
{\bf 54} (1990), 421--432.

 
\bibitem{bjw} W. Bosma, H. Jager and F. Wiedijk, Some metrical observations on the approximation
 of continued factions, {\em Indag. Math.} 45 (1983), 281--299.
 
 \bibitem{buser}
 P. Buser, {\em Geometry and     Spectra of Compact Riemann Surfaces }, Progress in Math. vol. 106, Birkha\"user, Boston 1992.
 
 
\bibitem{dk} K. Dajani and C. Kraaikamp, {\em Ergodic Theory of Numbers}, Carus Mathematical Monographs, 29. Mathematical Association of America, Washington, DC, 2002. 


\bibitem {sinai} I. P. Cornfeld,  S.V. Fomin and
Ya. G. Sinai, {\em Ergodic Theory,}  Springer-Verlag, Berlin-Heidelberg-New York, 1982. 



  \bibitem{haas3} A. Haas, The distribution of geodesic excursions out the end of a hyperbolic orbifold and 
  approximation with respect to a Fuchsian group, {\em Geom. Dedicata}, Volume 116, No. 1 (2005), 129-155

\bibitem{haas} A. Haas, Geodesic cusp excursions and metric diophantine approximation,  {\em Math. Res. Lett.}  {\bf16}  (2009),  no. 1, 67--85.
\bibitem{haas2} A. Haas, Geodesic excursions into an embedded disc on a hyperbolic Riemann surface, {\em Conform. Geom.   Dyn.}   {\bf 13} (2009), 1--5.  

 


\bibitem{nakada} H. Nakada, On metrical theory of Diophantine approximation over imaginary quadratic field.
{\em Acta Arith.} {\bf 51} (1988), no. 4, 393--403. 



\bibitem{nicholls} P. Nicholls, {\em The Ergodic Theory of Discrete Groups,} 
Cambridge Univ. Press, 1989.

 \bibitem{strat} B. Stratmann, A note on counting cuspidal excursions. {\em Ann. Acad. Sci. Fenn. Ser. A I Math.} {\bf 20} (1995) no.2, 359--372.

\bibitem{sullivan} D. Sullivan, Disjoint spheres, approximation by imaginary quadratic numbers and 
the logarithm law for geodesics. {\em Acta Math.} {\bf 149} (1982), 215--273.





\end{thebibliography}
\end{document}